\documentclass{amsart}
\usepackage{graphicx} 
\usepackage{amsmath}
\usepackage{amssymb}
\usepackage{amsthm}
\usepackage{color}
\usepackage{enumerate}
\usepackage{lineno}
 \numberwithin{equation}{section}
\theoremstyle{definition}
\newtheorem{thm}{Theorem}[section]
\newtheorem{prop}[thm]{Proposition}
\newtheorem{lemma}[thm]{Lemma}
\newtheorem{definition}[thm]{Definition}
\newtheorem{rem}[thm]{Remark}
\newtheorem*{pf}{Proof of Theorem 1.1}

\title[convergence rate at infinity for the cusp winding spectrum]{Polynomial convergence rate at infinity for the cusp winding spectrum of generalized Schottky groups}
\author{Yuya Arima}
 \address{Graduate School of Mathematics, Nagoya University,
Furocho, Chikusaku, Nagoya, 464-8602, Japan}
\email{e-mail: yuya.arima.c0@math.nagoya-u.ac.jp}

\begin{document}

\maketitle
\begin{abstract}
    We show that the convergence rate of the cusp winding spectrum to the Hausdorff dimension of the limit set of a generalized Schottky group with one parabolic generator is polynomial. Our main theorem provides the new phenomenon in which differences in the Hausdorff dimension of the limit set generated by a Markov system cause essentially different results on multifractal analysis. This paper also provides a new characterization of the geodesic flow on the Poinca\'re disc model of two-dimensional hyperbolic space and the limit set of a generalized Schottky group. To prove our main theorem we use thermodynamic formalism on a countable Markov shift, gamma function, and zeta function.     
\end{abstract}
\section{Introduction}
In this paper, we consider the Poincar\'e disc model $(\mathbb{D}, d)$ of the two-dimensional hyperbolic space. For the sake of simplicity, we postpone some technical definitions to the Section 2. Let $G$ be a generalized Schottky group with one parabolic generator generated by $G_0$. We can write $G_0=H_0\cup \Gamma_0$, where $H_0$ is {the non-empty finite} set of hyperbolic generators and $\Gamma_0:=\{\gamma^{\pm1}\}$ is the set of a parabolic generator with the fixed point $p$ which is in the Euclidean boundary of $\mathbb{D}$. Note that $G$ is a non-elementary finitely generated free Fuchsian group. We denote by $\Lambda(G)$ the limit set of $G$ and by $\Lambda_c(G)$ the conical limit set. 
In this setting, the conical limit set is given by 
\begin{align*}
   \Lambda_c(G)=\Lambda(G)\setminus\bigcup_{g\in G}\{g(p)\}. 
\end{align*}
Thus, the set $\Lambda(G)\setminus\Lambda_c(G)$ is a { countable }   set. Therefore, the Hausdorff dimension of $\Lambda_c(G)$ coincides with the Hausdorff dimension of $\Lambda(G)$. Note that we have $1/2<\dim_H(\Lambda_c(G))<1${, where $\dim_H$ denotes the Hausdorff dimension.}
(see Theorem \ref{Theorem Beardon})

{we recall the definition and the motivation of the cusp winding process from \cite{Jaerisch2016AMA} (see also \cite{jaerisch2021mixed} and \cite{arima2024higher}).} Let $R\subset \mathbb{D}$ be the Dirichlet fundamental domain for $G$ at centered $0$. For a conical limit point $x\in\Lambda_c(G)$ we can construct the unique infinite sequence $\omega(x)=\omega_0(x)\omega_1(x)\cdots$ in $G^{\mathbb{N}\cup\{0\}}$ which is associated to $x$ as follows: Consider the oriented geodesic ray $s_x$ from $0$ to $x$. The oriented geodesic $s_x$ intersects the infinitely many copies $R$, $g_0(x)R$, $g_0(x)g_1(x)R$,... of $R$, with $g_i(x)\in G_0$ and $i\in\mathbb{N}\cup\{0\}$. Thus, we obtain the infinite sequence $\omega(x)=g_0(x)g_1(x)g_2(x)\cdots\in G_0^{\mathbb{N}\cup\{0\}}$, which is necessarily reduced, that is, $g_{i-1}(x)g_i(x)\neq I$ for all $i\in\mathbb{N}${, where $I$ denotes the identity map}.
Using the infinite sequence $\omega(x)$ in $G^{\mathbb{N}\cup\{0\}}$, we can define a block sequence which is used to define the cusp winding process as follows: Let $x$ be a conical limit point. 
 We define a block sequence $B_i(x)$,  $i\in \mathbb{N}$, such that
\begin{align*}
\omega(x)=B_1(x)B_2(x)\cdots,    
\end{align*}
 where each $B_i(x)$ is either a hyperbolic generator, or a maximal block of consecutive appearances of the same parabolic generator.
 By construction, for $\gamma\in\Gamma_0$, $l\in\mathbb{N}\cup\{0\}$ and $i\in\mathbb{N}$ a block $B_i(x)=\gamma^{l+1}$ means that the projection of $s_x$ onto $\mathbb{D}/G$ winds
 $l$ times around the cusp $p$. Motivated by counting the number of windings around the cusp $p$, we define the cusp winding process
$(a_{i})_{i\geq1}:\Lambda_{\rm c}(G)\rightarrow \mathbb{N}\cup\{0\}$ by
\begin{align*}
    a_{i}(x)= \left\{
 \begin{array}{cc}
   l   & \text{if}\  B_i(x)=\gamma^{l+1}\ \text{or}\ B_i(x)=\gamma^{-(l+1)},\ l\geq1\\
   0   & \text{otherwise}
 \end{array}
 \right..
\end{align*}
For ${\alpha}\in[0,\infty]$ we define the level set by 
\begin{align*}
    J({\alpha}):=\left\{x\in\Lambda_{\rm c}(G):\lim_{n\to\infty}\frac{1}{n}\sum_{i=1}^{n}a_{i}(x)=\alpha\right\}.
\end{align*}
Using the above level set, we can consider the following multifractal decomposition of the conical limit set $\Lambda_c(G)$:
\begin{align*}
    \Lambda_{\rm c}(G)
=\left(\bigcup_{{\alpha}\in[0,\infty]}J({\alpha}) \right)\cup J_{\text{ir}},
\end{align*}
where $J_{\text{ir}}$ denotes the irregular set, that is, the set of conical limit points $x\in\Lambda_c(G)$ for which the limit $(1/n)\sum_{i=1}^{n}a_{i}(x)$ does not exist. To investigate the growth rate of the number of cusp winding by the Hausdorff dimension, we define the cusp winding spectrum as follows:
\begin{align*}
    b:[0,\infty]\rightarrow [0,\dim_H\Lambda_{\rm c}(G)],\ b({\alpha})=\dim_HJ({\alpha}).
\end{align*}
Also, the function $b$ is simply called the dimension spectrum.
Denote by $f$ the Bowen-Series map associated with the Dirichlet fundamental domain $R$ centered at $0$ and by $\tilde f$ the induced system derived from the Bowen-Series map $f$ as defined in Section 2.1. By the proof of \cite[Proposition 2.2]{Dal}, 
the maximal invariant subset of $\partial \mathbb{D}$
generated by the induced system $\tilde f$ is the conical limit set $\Lambda_{\rm c}(G)$. Thus, we can consider the dynamical system $(\tilde f, \Lambda_c(G))$. Moreover, by the definition of the induced system $\tilde f$ we have $a_1\circ \tilde f^{i-1}=a_i$ for all $i\in\mathbb{N}$ (see Section 2.1). Hence, the dimension spectrum $b$ can be regarded as a Birkhoff spectrum. We call the triple $(\tilde f, \Lambda_c(G), a_1)$ a generalized Schottky system.
The detailed analysis of the multifractal decomposition of the conical limit set and dimension spectra
is obtained from \cite{arima2024higher} (see also \cite{jaerisch2021mixed}). 

By \cite[Proposition 1.2]{arima2024higher}, we have that the dimension spectrum $b$ is strictly increasing and $\lim_{\alpha\to\infty}b(\alpha)=\dim_H(\Lambda_c(G))$. Therefore, it is natural to ask about the convergence rate of $b$ to the Hausdorff dimension of the conical limit set. The following theorem is our main theorem. 
\begin{thm}\label{Theorem main}
    Let $G$ be a generalized Schottky group. We have
    \begin{align}
 \lim_{\alpha\to\infty}(\dim_H({\Lambda_c(G)})-b(\alpha))\alpha^x= \left\{
 \begin{array}{cc}
   \infty   & \text{if}\  {1}/({2-2\dim_H{\Lambda_c(G)}})-1<x\\
   0   & \text{if}\   {1}/({2-2\dim_H{\Lambda_c(G)}})-1>x  \nonumber
 \end{array}
 \right. .        
    \end{align}
\end{thm}
Multifractal analysis has been studied in several settings. We refer the reader to \cite{barreira2008dimension} and \cite{pesin2008dimension} for basic results on multifractal analysis. By studies on multifractal analysis, it is known that the lack of certain conditions leads to strange results on multifractal analysis. For instance, the lack of compactness of a phase space can cause the existence of a point such that a Birhkoff spectrum is not analytic at this point (see \cite{iommi2015multifractal}). Also, the presence of a neutral fixed point can cause the phenomenon in which a Birhkoff spectrum is completely flat (see \cite{jaerisch2021mixed}). However, to our knowledge, there is no known result in which differences in the Hausdorff dimension of the limit set generated by a Markov system cause essentially different results on multifractal analysis. Comparing Theorem \ref{Theorem main} with the result regarding the convergence rate of the Birkhoff spectrum of the arithmetic mean of the continued fraction, we can see that Theorem \ref{Theorem main} exhibits a new phenomenon. To explain this, we introduce the Gauss system. 

Let $T:(0,1]\rightarrow (0,1]$ be the Gauss map defined by $T(x)=1/x-[1/x]$, where $[\cdot]$ denotes the floor function. We define intervals $(I_m)_{m\in\mathbb{N}}$ given by $I_m:=(1/(m+1),1/m]$ for $m\in\mathbb{N}$.  
The Gauss map $T$ is a Markov map with the countable Markov partition $(I_m)_{m\in\mathbb{N}}$ and the limit set generated by $T$ (i.e. $\cup_{(m_0,m_1,\cdots)\in\mathbb{N}^{\mathbb{N}\cup\{0\}}}\cap_{i=0}^\infty T^{-i}I_{m_i}$) is $I:=(0,1)\setminus \mathbb{Q}$. Note that the Hausdorff dimension of $I$ is $1$. We define the digit functions $(\mathfrak{d}_i)_{i\geq 1}$ as $\mathfrak{d}_i(x):=m$ if $T^{i-1}(x)\in I_m$ for $i,m\in\mathbb{N}$. The digit functions are closely related to the continued fraction for a irrational number in $(0,1)$. Note that for $i\in\mathbb{N}$ we have $\mathfrak{d}_i=\mathfrak{d}_1\circ T$ by definition of the digit functions $(\mathfrak{d}_i)_{i\geq 1}$. We call the triple $(T,I,\mathfrak{d}_1)$ the Gauss system.

The Gauss system $(T,I,\mathfrak{d}_1)$ is well-studied from various perspectives. See for example \cite{fan2009khintchine}, \cite{hensley1992continued}, \cite[Section 6]{iommi2015multifractal} and \cite{mauldin1999conformal}.    
 Let $n\in\mathbb{N}$.
 We consider the fully constrained set $E(n)$ by $n$, that is, the set of points $x$ in $I$ such that $\mathfrak{d}_i(x)$ does not exceed $n$ for all $i\in\mathbb{N}$. By \cite{hensley1992continued}, we can completely understand the asymptotic behavior of the function $n\mapsto \dim_H(E(n))$ as follows: 
\begin{align*}
    \dim_H(E(n)) = 1 -\frac{6}{\pi^2 n} - 72 \frac{\log n}{\pi^4n^2} + O\left(\frac{1}{n^2}\right)\ \text{as }n\to\infty,
\end{align*}
where $O$ denotes the Landau's notation. Next, for $\alpha\in[1,\infty)$ we consider the average constrained set $\tilde E(\alpha)$, that is, the set of points $x$ in $I$ such that the average $(1/n)\sum_{i=0}^{n-1}\mathfrak{d}_i(x)$ does not exceed $\alpha$ for all $n\in\mathbb{N}$. By \cite[Theorem 2]{cesaratto2006hausdorff}, the asymptotic behavior of the function $\alpha\mapsto\dim_H(\tilde E(\alpha))$ is also understood as follows:
\begin{align*}
    \dim_H(\tilde E(\alpha))=1-O\left(\left(\frac{1}{2}\right)^{\alpha}\right).
\end{align*}
For ${\alpha}\in[1,\infty]$ we define the level set by 
\begin{align*}
    \tilde J({\alpha}):=\left\{x\in I:\lim_{n\to\infty}\frac{1}{n}\sum_{i=1}^{n}\mathfrak d_{i}(x)=\alpha\right\}
\end{align*}
and the Birkhoff spectrum $\mathfrak{b}:[1,\infty]\rightarrow \mathbb{R}$ given by $\mathfrak{b}(\alpha):=\dim_H(\tilde J(\alpha))$.
We can also obtain the following proposition regarding the convergence rate of the Birkhoff spectrum of the arithmetic mean of the continued fraction $\mathfrak{b}$. 
\begin{prop}\label{prop Gauss}
    For the Gauss system $(T,I,\mathfrak{d}_1)$ and $\alpha\in(1,\infty)$ we 
    have $\mathfrak{b}(\alpha)=\dim_H(\tilde E(\alpha))$. Moreover, we have
\begin{align*}
\mathfrak{b}(\alpha)=1-O\left(\left(\frac{1}{2}\right)^{\alpha}\right).
\end{align*}    
\end{prop}
Proposition \ref{prop Gauss} is the immediately consequence from \cite[Lemma 4.4 and Proposition 4.6]{iommi2015multifractal} and \cite[Theorem 1]{cesaratto2006hausdorff}. 

Proposition \ref{prop Gauss} states that the convergence rate of the Birkhoff spectrum of the arithmetic mean of the continued fraction $\mathfrak{b}$ to the Hausdorff dimension of the limit set $I$ is exponential. By \cite{series1985modular}, the generalized Schottky system $(\tilde f, \Lambda_c(G), a_1)$ is analogous to the Gauss system $(T,I,\mathfrak{d}_1)$. 
Thus, if the convergence rate of $b$ to $\dim_H(\Lambda_c(G))$ and the convergence rate of $\mathfrak{b}$ to $\dim_H(I)=1$ are different, then the essential difference in the result of multifractal analysis is caused by the difference between the Hausdorff dimension of $I$ and the Hausdorff dimension of $\Lambda_c(G)$ (Note that $1/2<\dim_H(\Lambda_c(G))<1$). In fact, Theorem \ref{Theorem main} state that the convergence rate of the dimension spectrum $b$ to $\dim_H(\Lambda_c(G))$ is polynomial. Therefore, we can see the new phenomenon in which differences in the Hausdorff dimension of the limit set generated by a Markov system cause an essentially different result in multifractal analysis.\\

Next, we consider Theorem \ref{Theorem main} in terms of hyperbolic geometry. There are numerous studies on  the geodesic flow on hyperbolic surfaces using a Fucshian group. Especially, there are a lot of results on the geodesic flow on a hyperbolic surface and the limit set of a non-elementary finite generated Fuchsian group obtained by performing a multifractal analysis. However, for a non-elementary finite generated Fuchsian group $G$ the relationship between $\dim_H(\Lambda(G))$ and the geodesic flow on the hyperbolic surface obtained from these results is that the supremum of a dimension spectrum including information of the geodesic flow on the hyperbolic surface is $\dim_H(\Lambda(G))$. For example, by \cite{Jaerisch2016AMA}, we known that for a non-elementary finite generated free Fuchsian group $G$ with parabolic generators the maximum of the dimension spectrum describing the fluctuation of a certain asymptotic exponential scaling associated to number of winding around a cusp is $\dim_H(\Lambda(G))$. On the other hand, Theorem \ref{Theorem main} states that for a generalized Schottky group $G$  with one parebolic generator the convergence rate of the dimension spectrum to $\dim_H(\Lambda(G))$ is determined by the $\dim_H(\Lambda(G))$. Since the cusp winding spectrum $b$ includes geometric information of the geodesic flow on $\mathbb{D}/G$, this means that we can also relate geodesic flow on $\mathbb{D}/G$ and the Hausdorff dimension of $\Lambda(G)$ using the convergence rate of the cusp winding spectrum $b$ to $\dim_H(\Lambda(G))$. This is a new characterization of geodesic flow on the $\mathbb{D}/G$ and the limit set of a generalized Schottky group.\\

\textbf{Methods of proofs.}
To prove Theorem \ref{Theorem main}, we relate to thermodynamic formalism, which is also used in \cite{iommi2015multifractal}, with methods of proofs  used in \cite{cesaratto2006hausdorff}. We first define the key function as follows: 
\begin{align*}
 p:\mathbb{R}^3\rightarrow \mathbb{R},\ \ \    
p(\alpha,{q},b):=P({q}(-a_1+\alpha)-b\log|\tilde f'|),\  
\end{align*}
where $P$ denotes the topological pressure (see Definition \ref{Definition of pressure}). By \cite{arima2024higher},  there exists real-analytic function $\alpha\in(0,\infty) \mapsto q(\alpha)\in(0,\infty)$ such that \begin{align}\label{equation G}
    p({\alpha},{q}({\alpha}),b({\alpha}))=0\ \text{and}\  \frac{\partial}{\partial q}p({\alpha},{q}({\alpha}),b({\alpha}))=0.
\end{align}
Using (\ref{equation G}) and real-analyticity of the dimension spectrum $b$ (see Theorem \ref{Theorem main clain of previous paper}), for  $\alpha\in(0,\infty)$ large enough we can relate $\dim_H(\Lambda_c(G))-b(\alpha)$ with $\int_{\alpha}^\infty q(t) dt$. On the other hand, using (\ref{equation G}), Ruelle's formula for the derivative of the pressure, and the Gibbs property, for $\alpha\in(0,\infty)$ large enough we obtain a relationship between $\alpha$ and the Dirichlet series $\sum_{l=1}^\infty\exp(-lq(\alpha)){l^{1-2b(\alpha)}}$.  By Mellin transformation, Gamma function and zeta function, one can show a comparability between $\alpha^{1/(2-2b(\alpha))}$ and $q(\alpha)$ for large enough $\alpha\in(0,\infty)$. By those relations, we can show Theorem \ref{Theorem main}.\\  

\textbf{Plan of the paper.}

 In the Section \ref{Section Discrete geometry}, we introduce the precise definition on Discrete geometry which is used in this paper.

 In the Section \ref{Section Thermodynammic formalism}, we first describe thermodynamic formalism for Countable Markov Shift and a coding between Countable Markov Shift and the dynamical system $(\tilde f,\Lambda_c(G))$. Then, we explain thermodynamic formalism for the dynamical system $(\tilde f,\Lambda_c(G))$.

 In the Section \ref{Section Summary of the results}, we recall some results from \cite{arima2024higher}. This result play a fundamental rule in this paper.

 In the section \ref{section proof of the main theorem}, we first show some technical lemma. By these lemmas, we obtain the comparability between $\alpha^{1/(2-2b(\alpha))}$ and $q(\alpha)$ for large enough $\alpha\in(0,\infty)$. Then, using these lemmas, we prove Theorem \ref{Theorem main}.

\section{Preliminaries}
\subsection{The Bowen-Series map and multi-cusp winding process}\label{Section Discrete geometry}

{In this section, we will first give some definitions of hyperbolic geometry and the notation used throughout this paper.} We refer to the reader \cite{beardon2012geometry},\cite{Dal} and \cite{katok1992fuchsian} for details on discrete geometry. Let $(\mathbb{D},d)$ denote the Poincar\'{e} disc model of two-dimensional hyperbolic spaces. We denote by Conf$(\mathbb{D})$ the set of orientation-preserving isometries of $(\mathbb{D},d)$. Recall that, in this setting, Conf$(\mathbb{D})$ is the set of M\"{o}bius transformations (see \cite[Proposition 1.1]{Dal}) and each element $g$ of Conf$(\mathbb{D})$ is classified as hyperbolic, parabolic and elliptic using fix points of $g$. A element $h$ of Conf$(\mathbb{D})$ is called a hyperbolic element if $h$ has two fixed points in $\partial \mathbb{D}$ , where $\partial \mathbb{D}$ denotes the Euclidean boundary of $\mathbb{D}$. A element $\gamma$ of Conf$(\mathbb{D})$ is called a parabolic element if $\gamma$ has one fixed point in $\partial \mathbb{D}$. Note that for a parabolic element $\gamma\in G$ and its fixed point $p\in\partial \mathbb{D}$ we have $|\gamma'(p)|=1$, where $|\cdot|$ denotes the Euclidean metric norm of $\mathbb{R}^2$. A element $\phi$ of Conf$(\mathbb{D})$ is called a elliptic element if $\phi$ has  one fixed point in  $\mathbb{D}$.  A subgroup $G$ of Conf$(\mathbb{D})$ is called a Fuchsian group if $G$ is a discrete subgroup of Conf$(\mathbb{D})$. Let $G$ be a Fuchsian group. We define the limit set of $G$ by 
\begin{align*}
\Lambda(G):=\bigcup_{g\in G}\overline{\{g0\}},  
\end{align*}
where $\overline{\{g0\}}$ denotes the Euclidean closure of $\{g0\}$ for $g\in G$. $G$ is said non-elementary if the limit set of $G$ is not a finite set. A limit point $x\in \Lambda(G)$ is called a conical limit point if  there exists $(g_n)_{n\in\mathbb{N}}\subset G$ such that $\lim_{n\to\infty}g_n0=x$ and the sequence $(\inf_{z\in[0,x)}d(g_n0,z))_{n\in\mathbb{N}}$ is bounded, where $[0,x)$ denotes the geodesic ray connecting $0$ and $x$. We denote by $\Lambda_c(G)$ the set of conical limit points.\\

Next, we introduce the definition of a generalized Schottky group. For $g\in$Conf$(\mathbb{D})$ we define the isometry circle of $g$ as $\Delta(g):=\{z\in\partial\mathbb{D}:|g'(z)|\geq1\}$. Let $n,m\in\mathbb{N}$. Let $H_0:=\{h_1^{\pm1},\cdots,h_n^{\pm1}\}\subset$Conf$(\mathbb{D})$ be the set of hyperbolic generators and 
let $\Gamma_0:=\{\gamma_1^{\pm1},\cdots,\gamma_m^{\pm1}\}\subset$Conf$(\mathbb{D})$ be the set of parabolic generators. We put $G_0=H_0\cup\Gamma_0$ and write $G_0=\{g_1^{\pm1},\cdots,g_{n+m}^{\pm1}\}$. We assume that  $\overline{\Delta({g_i})\cup \Delta({g_i^{-1}})}\cap \overline{{\Delta({g_j})\cup \Delta({g_j^{-1}})}}= \emptyset$
for $i,j\in\{1,2,\cdots,m+n\}$ with $i\neq j$. Let $G$ be the subgroup of Conf$(\mathbb{D})\}$ generated by $G_0$. $G$ is called a generalized Schottky group with $m$ parabolic generators generated by $G_0$. Let $G$ be a generalized Schottky group with $m\geq1$ parabolic generators generated by $G_0$. Note that a generalized Schottky group is a non-elementary finite generated free Fuchsian group with respect to the generator $G_0$ (see \cite[Proposition 1.6]{Dal}). Furthermore, the Dirichlet fundamental domain of $G$ centered at $0$ is given by 
\begin{align*}
    \bigcap_{g\in G_0}\{z\in\mathbb{D}:|g'(z)|<1\}.
\end{align*}
Since the generalized Schottky group $G$ is finite generated, the conical limit set $\Lambda_c(G)$ is given by the following form (see \cite[Theorem 10.2.5]{beardon2012geometry}):
\begin{align}\label{eq decomposition of the limit set}
    \Lambda_c(G)=\Lambda(G)\setminus\bigcup_{g\in G}\bigcup_{i=1}^m\{g(p_i)\},
\end{align}
where $p_i\in\partial \mathbb{D}$ denotes the fixed point of the parabolic generator $\gamma_i\in G_0$ for $i\in\{1,\cdots, m\}$. Thus, the set $\Lambda(G)\setminus\Lambda_c(G)$ is a countable set. This implies that the Hausdorff dimension of $\Lambda_c(G)$ is equal to the Hausdorff dimension of $\Lambda(G)$. For a subset $A$ of $\partial\mathbb{D}$ we denote by $\dim_H(A)$ the Hausdorff dimension of $A$. We have the following fundamental fact.
\begin{thm}\cite[Theorem 14.3]{borthwick2007spectral}\label{Theorem Beardon}
For a non-elementary finite generated Fuchsian group $G$ containing parabolic elements we have $1/2<\dim_H(\Lambda_c(G))<1$.
\end{thm}

In this paper, we always assume that $G$ is a generalized Schottky group with one parabolic generator generated by $G_0$. We recall definitions of the Bowen-Series map with respect to the Dirichlet fundamental domain $R$ centered at $0$ and the cusp-winding process. To do this, we put 
    $\Delta:=\bigcup_{g\in G_0}\Delta(g).$
\begin{definition}\cite{bowen1979markov}
    The Bowen-Series map with respect to the Dirichlet fundamental domain $R$ centered at $0$  is given by 
\begin{align}
   f:\Delta \rightarrow \partial \mathbb{D}, \ f|_{\Delta({g})}=g\ \ (g\in G_0).\nonumber
\end{align}
\end{definition}
The Bowen-Series map $f$ with respect to the Dirichlet fundamental domain $R$ centered at $0$ is simply called the Bowen-Series map.
Since the definition of a generalized Schottky group, for $h,\tilde h \in H_0$ we have $\Delta(h)\cap \Delta(\tilde h)=\emptyset$ and for $\gamma\in\Gamma_0$ we have $\Delta(\gamma^{-1})\cap\Delta(\gamma)=\{p\}$, where $p$ is the fixed point of $\gamma\in\Gamma_0$. Thus, the definition of the Bowen-Series map is well-defined. Also, note that  $\Lambda(G)$ is a $f$-invariant set.
\begin{rem}
By the choice of the fundamental domain, there exist constants $W>Z>1$ such that for all $h\in H_0$ and $x\in\Delta(h)\cap\Lambda(G)$ we have
    \begin{align}\label{eq uniformaliy bound }
        Z\leq |f'(x)|\leq W.
    \end{align}
\end{rem}
  For all $x\in \Lambda_c(G)$ there uniquely exists  $\omega(x)=\omega_0\omega_1\cdots \in G_0^{\mathbb{N}\cup\{0\}}$ such that $f^n(x)\in\Delta(\omega_n)$ and if $\omega_n$ is parabolic for some $n\in\mathbb{N}\cup\{0\}$ then there exists $m\in\mathbb{N}$ such that $m>n$ and $\omega_m\neq\omega_n$. Then, $\omega(x)$ defines a sequence of blocks $B_i$ ($i\in \mathbb{N}$) such that
$\omega(x)=B_1(x)B_2(x)\cdots$,
 where each $B_i(x)$ is either a hyperbolic generator, or a maximal block of consecutive appearances of the same parabolic generator.

\begin{definition}\cite{Jaerisch2016AMA}
    The cusp winding process is given by
\begin{align} 
(a_{i})_{i\geq1}:\Lambda_c(G)\rightarrow \mathbb{N}\cup\{0\},\  
 a_{i}(x)= \left\{
 \begin{array}{cc}
   m   & \text{if}\  B_i(x)=\gamma^{{m+1}} (\gamma \in \Gamma_0 )\\
   0   & \text{otherwise}\nonumber
 \end{array}
 \right. .
\end{align}
\end{definition}
Next, we describe the definition of the induced system derive from the Bowen-Series map. We define
\begin{align}
\mathcal{A}:=\bigcup_{l=1}^{\infty} \lbrace \gamma^{l}h:\gamma \in \Gamma_0,h \in H_0\rbrace \cup H_0\nonumber
\end{align}
and  the set
$\Delta({\omega}):=\Delta({\omega_0}) \cap f^{-1}\Delta({\omega_1}) \cap \cdots \cap f^{-(n-1)}\Delta({\omega_{n-1}})$ for $\omega=\omega_0\cdots\omega_{n-1} \in \mathcal{A}$ and $n\in\mathbb{N}$.

\begin{definition}
Define the inducing time $\tau:\mathcal{A}\rightarrow \mathbb{N}$ by $\tau(\gamma^{l+1}h)={l+1}$ ($\gamma \in \Gamma_0,\ h\in H_0$ and $l\in\mathbb{N}$) and $\tau|_{H_0}=1$. The induced Markov map with the Markov partition $\lbrace \Delta({\omega}) \rbrace_{\omega \in A}$ is given by
\begin{align}
\tilde f: \bigcup_{\omega \in A} \Delta(\omega)  \rightarrow \partial \mathbb{D}, \ \ \tilde f|_{\Delta({\omega})}=f^{\tau(\omega)}. \nonumber
\end{align}
\end{definition}

Note that  $a_{i}=a_{1}\circ \tilde f^{i-1} $ for $i\geq 1$ and the maximal $\tilde f$-invariant set is the conical limit set $\Lambda_c(G)$ by (\ref{eq decomposition of the limit set}) (see the proof of \cite[Proposition 2.2]{Dal}). In this paper, we denote simply $\tilde f|_{\Lambda_c(G)}$ as $\tilde f$ and consider the pair $(\tilde f, \Lambda_c(G))$ as a dynamical system.

\begin{rem}\label{Remark uniformaly expanding}
   Since for the fixed point $p\in\Lambda(G)$ of $\gamma\in\Gamma_0$ we have $f(p)=p$ and $|f'(p)|=|\gamma'(p)|=1$ (i.e. $p$ is a neutral fixed point of $f$), $f$ is not uniformaly expanding. But, for all $h\in H_0$ and 
all $\gamma\in\Gamma_0$ we have $\lim_{l\to\infty}\inf\{|\tilde f'(x)|:x\in\Delta(\gamma^lh)\}=\infty$ (see \cite[Lemma 4.1]{arima2024higher}). Therefore, by (\ref{eq uniformaliy bound }), $\tilde f$ is uniformaly expanding. 
\end{rem}


     For $\alpha\in [0,\infty]$ we define the level sets by
    \begin{align}
    J({\alpha}):=\left\{ x \in \Lambda_c(G):\lim_{n\to \infty}\frac{1}{n}\sum_{i=1}^{n}a_{i}(x) = {\alpha} \right\},\nonumber
    \end{align}
    and the dimension spectrum 
\begin{align}
   b:[0,\infty]\rightarrow\mathbb{R},\ \ b({\alpha}):=\dim_HJ({\alpha}).\nonumber
\end{align}

 \subsection{Thermodynamic formalism}\label{Section Thermodynammic formalism}
In this section, we describe the thermodynamic formalism. For details on the thermodynamic formalism we refer to the reader \cite[Section 2]{mauldin2003graph} and \cite[section 17]{urubanskinoninvertible}.

Recall that $\tilde f $ is a Markov map. Thus, $\tilde f $ determines a $\mathcal{A}\times \mathcal{A}$ matrix $A$ by $A_{a,b}=1$ if $\Delta_b\subset\tilde f \Delta_a$ and $A_{a,b}=0$ otherwise. Define
\begin{align}
    \Sigma_A:=\lbrace \omega \in \mathcal{A}^{\mathbb{N}\cup\{0\}}
    :A_{\omega_{n-1},\omega_{n}}=1,\ n\in \mathbb{N} \rbrace\nonumber.
\end{align}
 A string $(\omega_0,\omega_1,\ldots,\omega_{n-1})\in\mathcal{A}^n$ is called an admissible word of length $n$ if $A_{\omega_{i-1},\omega_{i}}=1$ for all $i=0,\ldots,n-1$. 
We denote by $E^n$ the set of all admissible words of length $n$ for $n\in\mathbb{N}$
and by $E^*$ the set of all admissible words which have a finite length (i.e. $E^*=\cup_{n\in\mathbb{N}}E^n$).
For convenience, put $E^0=\lbrace \emptyset \rbrace$.
For $\omega \in E^n$ we define the cylinder set of $\omega$ by $[\omega]:=\lbrace \tau \in \Sigma_A:\tau_i=\omega_i, 0\leq i\leq n-1\rbrace$. 
Note that $\Sigma_A$ is finitely primitive since $hg\in E^*$, $gh\in E^*$ for all $h\in H_0$ and $g\in\mathcal{A}\setminus\{h^{-1}\}$. 

We endow $\Sigma_A$ with the topology generated by the cylinders. 
Since $\Sigma_A$ is finitely primitive and $\mathcal{A}$ is not a finite set, $\Sigma_A$ is not locally compact.
\begin{definition}
    We define the shift map $\sigma:\Sigma_{A}\rightarrow\Sigma_{A}$ by 
    \begin{align*}
        \sigma((\omega_0,\omega_1,\omega_2\cdots))=(\omega_1,\omega_2,\cdots)\ \text{for}\ (\omega_0,\omega_1,\omega_2\cdots)\in\Sigma_{A}
    \end{align*}
\end{definition}
Since $\tilde f$ is uniformaly expanding and for each $\omega\in\mathcal{A}$ the set $\Delta(\omega)$ is a compact set, for any $\omega=(\omega_0,\omega_1,\cdots) \in \Sigma_A$ the set 
    $\bigcap_{j=0}^{\infty}{\tilde{f}^{-j}\Delta(\omega_j)}\nonumber$
    is a singleton.
Thus, we can define the coding map $\pi:\Sigma_A\rightarrow \pi(\Sigma_A)$ given by 
\begin{align}
    \pi(\omega)\in\bigcap_{n=0}^\infty{\tilde{f}^{-j}\Delta(\omega_j)},\ \ \ (\omega=(\omega_0,\omega_1,\cdots) \in \Sigma_A)\nonumber.
\end{align}
\begin{rem}\label{Remark isomorphism}
    Since for all $\omega,\tau\in\mathcal{A}$ we have $\Delta(\omega)\cap\Delta(\tau)=\emptyset$, the coding map $\pi$ is homeomorphism. Moreover, the coding map $\pi$ satisfies $\tilde{f}(\pi(\omega))=\pi(\sigma(\omega))$ for $\omega \in \Sigma_A$ and $\pi(\Sigma_A)=\Lambda_{\rm c}(G)$. Thus, we can consider the dynamical system $(\Sigma_A,\sigma)$ instead of $(\tilde f,\Lambda_c(G))$.
\end{rem}

A function $\phi$ on $\Sigma_A$  is called weakly H\"older if there exist $Z>0$ and $t\in(0,1)$ such that 
\begin{align*}
    \sup \lbrace |\phi(\omega)-\phi(\tau)|:\omega,\tau\in\Sigma_A,\ \omega_i=\tau_i\ \text{for}\ 0\leq i\leq n-1\rbrace\leq Zt^n.
\end{align*}
\begin{rem}
     Since $\tilde f$ is uniformaly expanding (see Remark \ref{Remark uniformaly expanding}), $\log|\tilde f'\circ \pi|$ (see \cite[Lemma 4.2.2]{mauldin2003graph}) is weakly H\"older. Since for all $n\geq 2$, $\omega\in E^n$ and $\tau_1,\tau_2\in[\omega]$ we have $a_1\circ\pi(\tau_1)=a_1\circ\pi(\tau_2)$, the cusp winding process $a_1\circ\pi$ is weakly H\"older. Thus, for all $(x,y)\in\mathbb{R}^2$ the potential $xa_1\circ\pi+y\log|\tilde f'\circ \pi|$ is also weakly H\"older. 
\end{rem}
 \begin{definition}[{\cite{mauldin2003graph}}]\label{Definition of pressure}
Let $\phi:\Sigma_A\rightarrow \mathbb{R}$ be a continuous function. The topological pressure of $\phi$ is defined by 
 \begin{align}
     P(\phi):=\lim_{n\to\infty}\frac{1}{n}\log \sum_{\omega\in E^n}\exp\left(\sup_{\tau\in[\omega]}\sum_{j=0}^{n-1}\phi(\sigma^i(\tau))\right)\nonumber.
 \end{align}
\end{definition}
The Topological pressure satisfies the following variational principle. 
\begin{thm}[{\cite[Theorem 2.1.8]{mauldin2003graph}}]
Let $\phi$ be a continuous function on $\Sigma_{A}$. We have 
    $P(\phi)=\sup \left\{h(\mu)+\int \phi d\mu \right\},\nonumber$
    where the supremum is taken over all $\sigma$-invariant ergodic Borel probability measures $\mu$ supported by $\Sigma_A$ satisfying $\int \phi d\mu>-\infty$ and $h(\mu)$ is the measure-theoretic entropy with respect to $\sigma$.
\end{thm}
Let $\phi$ be a continuous function on $\Sigma_{A}$.
    A $\sigma$-invariant Borel probability measure $\mu$ supported by $\Sigma_A$ is called an equilibrium measure for $\phi$ if 
         $P(\phi)= h(\mu)+\int \phi d\mu\nonumber.$
\begin{definition}
    A Borel probability measure $\mu$ supported by $\Sigma_A$ is called a Gibbs measure for the potential $\phi$ if there exists a constant $M>1$ such that for all cylinder $[\omega]\ (\omega \in E^n,\ n\in\mathbb{N})$ and all $\tau\in [\omega]$ we have 
    \begin{align}
        \frac{1}{M}\leq\frac{\mu([\omega])}{\exp(-nP(\phi)+\sum_{j=0}^{n-1}\phi(\sigma(\tau)))}\leq M\nonumber.
    \end{align}
\end{definition}
Furthermore, the topological pressure satisfies the following basic and a important result.
\begin{thm}[{\cite[Theorem 2.2.9]{mauldin2003graph}}]\label{Theorem Gibbs measure}
Let $\phi:\Sigma_{A}\rightarrow\mathbb{R}$ be a weakly H\"older function. 
 Suppose that $\phi$ satisfies $P(\phi)<\infty$ and
     $\sum_{\omega\in \mathcal{A}}\inf(-\phi|_{[\omega]})\exp(\inf\phi|_{[\omega]})<\infty\nonumber$
 (i.e. $\int -\phi d\mu<\infty$ for all Gibbs measure $\mu$). Then, there exists a unique equilibrium measure $\mu_\phi$ such that 
      $\mu_{\phi}$ is a Gibbs measure.
 \end{thm}
Next, we describe the thermodynamic formalism on the dynamical system $(\tilde f,\Lambda_{\rm c}(G))$.
\begin{definition}
The topological pressure of a continuous function $\phi : \Lambda_{\rm c}(G)\rightarrow \mathbb{R}$ is defined by
$P_{\tilde f}(\phi)=\sup \left\{ h(\mu) + \int \phi d\mu \right\},\nonumber$ where the supremum is taken over all $\tilde f$-invariant ergodic Borel probability measures $\mu$ supported by $\Lambda_c(G)$ satisfying $\int \phi d\mu>-\infty$.
\end{definition}
By Remark \ref{Remark isomorphism}, there exists a bijection between the set of $\sigma$-invariant Borel probability measures supported by $\Sigma_A$ and the sat of $\tilde f$-invariant Borel probability measures supported by $\Lambda_c(G)$. Thus, for a continuous function $\phi$ on $\Lambda_c(G)$ we obtain $P_{\tilde f}(\phi)=P(\phi \circ \pi)$.  We will denote both pressures by $P$.\\

\section{Multifractal analysis for the cusp winding process}\label{Section Summary of the results}
In this section, we introduce some results from \cite{arima2024higher}.

Recall that $p(\alpha,q,b)=P(q(-a_1+\alpha)-b\log|\tilde f'|)$ for $(\alpha,q,b)\in\mathbb{R}^3$.
Let $\mathcal{F}:=\{({q},b)\in\mathbb{R}\times(0,\infty):P(-qa_1-b\log|\tilde f'|)<\infty\}$. 
By the following lemma, we can determine the set on which the function $(\alpha,q,b)\in\mathbb{R}^3\rightarrow p(\alpha,q,b)\in\mathbb{R}$ is finite.

\begin{lemma}\cite[Lemma 4.1]{arima2024higher}\label{Lemma convergence}
    We have $\mathcal{F}=(0,\infty)\times[0,\infty)\cup\{0\}\times(1/2,\infty)$.
\end{lemma}
By Lemma \ref{Lemma convergence} and Ruelle's formula (see \cite[Proposition 2.6.13]{przytycki2010conformal}), for $\alpha,q,b\in(0,\infty)$ we have the following useful formulas:
\begin{align}\label{eq Ruelle}
    \frac{\partial}{\partial q}p(\alpha,q,b)=\int (-a_1+\alpha)d\mu_{\alpha,q,b} \ \text{and } \frac{\partial}{\partial b}p(\alpha,q,b)=-\int\log|\tilde f'| d\mu_{\alpha,q,b},
\end{align}
where $\mu_{\alpha,q,b}$ denotes the equilibrium measure of the potential $q(-a_1+\alpha)-b\log|\tilde{f'}|$. On the other hand, using Lemma \ref{Lemma convergence}, we obtain Bowen's formula.
\begin{thm}\label{Theorem useful Bowen}
    We have $P(-\dim_H(\Lambda_c(G))\log|\tilde f'|)=0$.  
\end{thm}

The following lemma will be used later.
\begin{lemma}\cite[Lemma 4.1]{arima2024higher}\label{almost ^2}
    There exists a constant $K>1$ such that 
\begin{align}\nonumber
    &-2t(\log  K + \log l)
    \leq\inf_{ [\gamma^lh]}(-t\log |(\gamma^l)'| \circ \pi)
    \\&\leq \sup_{ [\gamma^lh]}(-t\log |(\gamma^l)'| \circ \pi)
    \leq -2t(-\log  K + \log l)\nonumber
\end{align}
for all $l\geq 1$, $t\leq0$, $\gamma\in \Gamma_0,$ and $h\in H_0$.
\end{lemma}
To prove the main theorem, the following proposition is important.
\begin{prop}\cite[Proposition 5.4]{arima2024higher}\label{prop more important proposition}
    For all $\alpha\in(0,\infty)$ there exists $q({\alpha})\in(0,\infty)$ such that 
\begin{align}\nonumber
     p(\alpha,q(\alpha),b(\alpha))=0\ \ \ 
        \text{and}\ \ \ 
\frac{\partial}{\partial q}p(\alpha,q({\alpha}),b(\alpha))=0.
\end{align}
Moreover, we obtain 
    $b(\alpha)={h(\mu_{q(\alpha)})}/{\lambda(\mu_{q(\alpha)})}$,
where $\mu_{q(\alpha)}$ denotes the equilibrium state of the potential $q(\alpha)(-a_1+\alpha)-b({\alpha})\log|\tilde f'|$.
\end{prop} 
By the proof of \cite[Proposition 6.2]{arima2024higher} and \cite[Theorem 1.1 and Proposition 1.2]{arima2024higher} we obtain the following theorem.
\begin{thm}\label{Theorem main clain of previous paper}
    The functions $\alpha\mapsto b(\alpha)$ and $\alpha\mapsto q(\alpha)$ are real-analytic on $(0,\infty)$. Moreover, $b(\alpha)$ is strictly increasing and we have $\lim_{\alpha\to\infty}b(\alpha)=\dim_H(\Lambda_c(G))$.
\end{thm}

\section{Proof of the main theorem}\label{section proof of the main theorem}
We will use the notations which is used in the previous section. Put $s:=\dim_H(\Lambda_c(G))$.
We consider the asymptotic behavior of the function $\alpha\mapsto q(\alpha)$ when $\alpha$ goes to $\infty$. 
\begin{lemma}\label{lemma asymptotic behavior of q}
    We have $\lim_{\alpha\to\infty}q(\alpha)=0$.
\end{lemma}
\begin{proof}
    Put $q_{\infty}:=\limsup_{\alpha\to\infty}q(\alpha)$. For a contradiction we assume that $0<q_{\infty}<\infty$. Then, there exists a strictly increasing sequence  $\{\alpha_{n}\}_{n\in\mathbb{N}}\subset(0,\infty)$ such that $\lim_{n\to\infty}\alpha_{n}=\infty$ and $\lim_{n\to\infty}q(\alpha_{n})=q_{\infty}$. By Proposition \ref{prop more important proposition}, we have $P(-q(\alpha_{n})a_1-b(\alpha_n)\log|\tilde f'|)=-q(\alpha_n)\alpha_n$ for all $n\in\mathbb{N}$. Therefore, $\lim_{n\to\infty}P(-q(\alpha_{n})a_1-b(\alpha_n)\log|\tilde f'|)=\lim_{n\to\infty}(-q(\alpha_n)\alpha_n)=-\infty$. By Proposition \ref{prop more important proposition}, we have $q(\alpha_n)\in(0,\infty)$ for all $n\in\mathbb{N}$. Hence, by Lemma \ref{Lemma convergence}, Theorem \ref{Theorem main clain of previous paper} and continuity of the topological pressure, we obtain $\lim_{n\to\infty}P(-q(\alpha_{n})a_1-b(\alpha_n)\log|\tilde f'|)=P(-q_{\infty}a_1-s\log|\tilde f'|)\in\mathbb{R}$. This is a contradiction.
    
    For a contradiction we assume that $q_{\infty}=\infty$. Then, there exists a strictly increasing sequence  $\{\alpha_{n}\}_{n\in\mathbb{N}}\subset(0,\infty)$ such that $\lim_{n\to\infty}\alpha_{n}=\infty$ and $\lim_{n\to\infty}q(\alpha_{n})=\infty$. By proposition \ref{prop more important proposition}, we have $\lim_{n\to\infty}p(\alpha_n,q(\alpha_n),b(\alpha_n))=0$. On the other hand, by the variational principle for the topological pressure, for all $n\in\mathbb{N}$ we obtain $p(\alpha_n,q(\alpha_n),b(\alpha_n))
    \geq \int (q(\alpha_n)(-a_1+\alpha_n)-b(\alpha_n)\log|\tilde f'|)d\delta_{x_{h(1)}}
    =q(\alpha_n)\alpha_n-b(\alpha_n)\log|\tilde f'(x_{h(1)})|$, where $x_{h(1)}:=\pi(h_1h_1\cdots)$ and $\delta_{x_{h(1)}}$ denotes the point mass measure at $x_{h(1)}$. Since $\lim_{n\to\infty}q(\alpha_n)\alpha_n=\infty$, we obtain $\lim_{n\to\infty}p(\alpha_n,q(\alpha_n),b(\alpha_n))=\infty$. This is a contradiction. Thus, since $q(\alpha)\in(0,\infty)$ for all $\alpha\in(0,\infty)$, we obtain $\lim_{\alpha\to\infty}q(\alpha)=0$.
    \end{proof}
    
    \begin{lemma}\label{Lemma the relationship q and alpha}
    We have $\lim_{\alpha\to\infty}q(\alpha)\alpha=0$.    
    \end{lemma}
    \begin{proof}
    By Theorem \ref{Theorem main clain of previous paper} and Theorem \ref{Theorem Beardon}, there exists $M>0$ such that for all $\alpha\geq M$ we have $b(\alpha)>1/2$. By Proposition \ref{prop more important proposition}, we have $-q(\alpha)\alpha=P(-q(\alpha)a_1-b(\alpha)\log|\tilde f'|)$ for all $\alpha \in (M,\infty)$. Therefore, by Lemma \ref{lemma asymptotic behavior of q}, Theorem \ref{Theorem useful Bowen} and continuity of the topological pressure, we obtain 
    $\lim_{\alpha\to\infty}(-q(\alpha)\alpha)
    =\lim_{\alpha\to\infty}P(-q(\alpha)a_1-b(\alpha)\log|\tilde f'|)
    =P(-s\log|\tilde f'|)=0.$  
    \end{proof}
  The asymptotic behavior of the function $\alpha \mapsto b(\alpha)$ when $\alpha$ goes to $\infty$ is associated with the asymptotic behavior of the function $\alpha\mapsto q(\alpha)$ when $\alpha$ goes to $\infty$.
\begin{lemma}\label{Lemma relationship q and b}
    There exist constants $B\geq1$ and $W>0$ such that for all $\alpha\in(W,\infty)$ we have
    \begin{align*}
        \frac{1}{B}\leq \frac{s-b(\alpha)}{\int_\alpha^\infty q(t)dt}\leq B.
    \end{align*}
    \end{lemma}    
\begin{proof}
    By Proposition \ref{prop more important proposition}, for all $\alpha\in(0,\infty)$ we have $\exp(P(-q(\alpha)a_1-b(\alpha)\log|\tilde f'|))=\exp(-\alpha q(\alpha))$. Differentiating this equation with respect to $\alpha$ and using (\ref{eq Ruelle}), we obtain 
       $(-b'(\alpha)\lambda(\mu_{\alpha})-\alpha q'(\alpha))\exp\left(-\alpha q(\alpha)\right)
       =(-q(\alpha)-\alpha q'(\alpha))\exp\left(-\alpha q(\alpha)\right)$
    and thus, $b'(\alpha)
        =q(\alpha)/\lambda(\mu_\alpha)$,
    where $\mu_{\alpha}$ denotes the equilibrium measure of the potential $-q(\alpha)(-a_1+\alpha)-b(\alpha)\log|\tilde f'|$.
    Therefore, by Theorem \ref{Theorem main clain of previous paper}, for all $\alpha\in(0,\infty)$ we obtain
    \begin{align}\label{eq relationship b and q}
        s-b(\alpha)=\int_{\alpha}^\infty\lambda(\mu_t)^{-1} q(t)dt.
    \end{align}
    Next, we show that there exist constants $C\geq 1$ and $W>0$ such that for all $\alpha\in( W,\infty)$ we have $\lambda(\mu_{\alpha})\leq C$.  By Theorem \ref{Theorem main clain of previous paper} and  Lemma \ref{Lemma the relationship q and alpha} and Theorem \ref{Theorem Gibbs measure}, for all $\alpha\in (0,\infty)$ we have $p(\alpha,q(\alpha),b(\alpha))=0$, $\lim_{\alpha\to\infty}b(\alpha)=s$, $\lim_{\alpha\to\infty}\alpha q(\alpha)=0$ and $\mu_{\alpha}$ be a Gibbs measure for the potential $-q(\alpha)(-a_1+\alpha)-b(\alpha)\log|\tilde f'|$. Therefore, by \cite[Theorem 17.6.4]{urubanskinoninvertible}, there exist  constants $C\geq 1$ and $ W_1>0$ such that for all $\alpha\in( W_1,\infty)$ and $\omega\in\mathcal{A}$ we have
    \begin{align*}
        \frac{\mu_\alpha(\pi[\omega])}{\exp(\sup_{\pi([\omega])}\{-q(\alpha)a_1-b(\alpha)\log|\tilde f'|\})}\leq C.
        \end{align*}
    By Theorem \ref{Theorem Beardon}, we can take a small $\epsilon\in(0,1)$ such that $1/2<s-\epsilon$. By Lemma \ref{lemma asymptotic behavior of q}
    and Theorem \ref{Theorem main clain of previous paper}, there exists a constant $W_2>W_1$ such that for all $\alpha\in(W_2,\infty)$ we have $0<q(\alpha)<1$ and $s-b(\alpha)<\epsilon$.
    By Lemma \ref{almost ^2}, there exist constants $M\geq1$ and $W>W_2$ such that for all $\alpha\in(M,\infty)$, $h\in H_0$, $\gamma\in \Gamma_0$ and $l\geq 1$ we have
    \begin{align*}
         \frac{\sup_{\pi([\gamma^lh])}\{\log|\tilde f'|\}}{\log l^2}\leq M
    \end{align*}
    and \begin{align*}
        \frac{\exp\left(\sup_{\pi([\gamma^lh])}\{-b(\alpha)\log|\tilde f'|\}\right)}{l^{-2b(\alpha)}}\leq M.
    \end{align*}
    Thus, since $s-\epsilon>1/2$, for all $\alpha\in(W,\infty)$ we obtain 
    \begin{align*}
    &\lambda(\mu_\alpha)
    =\sum_{h\in H_0,\gamma\in\Gamma_0}\sum_{l=0}^\infty\int_{\pi([\gamma^lh])}\log|\tilde f'|d\mu_{\alpha} 
    \\&\leq \sum_{h\in H_0,\gamma\in\Gamma_0}\sum_{l=0}^\infty\sup_{\pi([\gamma^lh])}\{\log|\tilde f'|\}\mu_{\alpha}(\pi([\gamma^lh]))
    \\&\leq C\sum_{h\in H_0,\gamma\in\Gamma_0}\sum_{l=0}^\infty\sup_{\pi([\gamma^lh])}\{\log|\tilde f'|\}\exp(-q(\alpha)l)\exp\left(\sup_{\pi([\gamma^lh])}\{-b(\alpha)\log|\tilde f'|\}\right)
    \\&\leq CM^2\sum_{h\in H_0,\gamma\in\Gamma_0}\sum_{l=0}^\infty(2\log l)\exp(-q(\alpha)l)l^{-2b(\alpha)}
    \\&
    \leq CM^2\sum_{h\in H_0,\gamma\in\Gamma_0}\sum_{l=0}^\infty(2\log l)l^{-2(s-\epsilon)}<\infty.
    \end{align*}
    {Since $\tilde f$ is uniformly expanding, there exists $c>0$ such that for all $\alpha\in(0,\infty)$ we have $c<\lambda(\mu_\alpha)$ . Therefore, the proof is complete.}
\end{proof}
The following lemma is shown by a standard argument using the Gibbs property, but it connects  \cite{iommi2015multifractal} and  \cite{cesaratto2006hausdorff}. 
\begin{lemma}\label{Lemma Dirichlet seires}
     There exist constants $C\geq 1$ and $Y>0$ such that for all $\alpha\in(Y,\infty)$ we have 
    \begin{align*}
        \frac{1}{C}\leq\frac{\sum_{l=1}^\infty\exp(-lq(\alpha)){l^{1-2b(\alpha)}}}{\alpha}<C
    \end{align*}
\end{lemma}

    \begin{proof}
        By Proposition \ref{prop more important proposition}, for all $\alpha\in(0,\infty)$ we have $(\partial/\partial q)p(\alpha,q(\alpha),b(\alpha))=\int (-a_1+\alpha)d\mu_{\alpha}=0$, where $\mu_{\alpha}$ denotes the equilibrium measure of the potential $q(\alpha)(-a_1+\alpha)-b(\alpha)\log|\tilde f'|$. Thus, we have $\int a_1d\mu_\alpha=\alpha$. Repeating the argument in the proof of Lemma \ref{Lemma relationship q and b}, there exist a constants $R,V\geq 1$ and $Y>0$ such that for all $\alpha\in(Y,\infty)$ and $\omega\in\mathcal{A}$ we have
    \begin{align*}
        \frac{1}{R}\leq\frac{\mu_\alpha(\pi[\omega])}{\exp(\sup_{\pi([\omega])}\{-q(\alpha)a_1-b(\alpha)\log|\tilde f'|\})}\leq R
    \end{align*}
    and \begin{align*}
        \frac{1}{V}\leq \frac{\exp\left(\sup_{\pi([\gamma^{l+1}h])}\{-b(\alpha)\log|\tilde f'|\}\right)}{l^{-2b(\alpha)}}\leq V.
    \end{align*}
    Thus, for all $\alpha\in(Y,\infty)$ we obtain
    \begin{align*}
        &\alpha=
        \int a_1d\mu_{\alpha}
        =\sum_{h\in H_0,\gamma\in\Gamma_0}\sum_{l=1}^\infty l\mu_{\alpha}(\pi([\gamma^{l+1}h]))
        \\& \leq R\sum_{h\in H_0,\gamma\in\Gamma_0}\sum_{l=1}^\infty l\exp\left(\sup_{\pi([\gamma^{l+1}h])}\{-q(\alpha)l-b(\alpha)\log|\tilde f'|\}\right)
        \\&\leq RV\sum_{h\in H_0,\gamma\in\Gamma_0}\sum_{l=1}^\infty \exp(-q(\alpha)l)l^{1-2b(\alpha)}.
    \end{align*}
    Repeating the above argument, for all $\alpha\in(Y,\infty)$ we obtain
    \begin{align*}
        \alpha=
        \int a_1d\mu_{\alpha}
        \geq \frac{1}{RV}\sum_{h\in H_0,\gamma\in\Gamma_0}\sum_{l=1}^\infty \exp(-q(\alpha)l)l^{1-2b(\alpha)}.
        \end{align*}
        Since $H_0$ and $\Gamma_0$ are finite sets, the proof is complete.
    \end{proof}
    We denote by $\Gamma$ the Gamma function and by $\zeta$ the zeta function.
\begin{prop}\label{proposition main}
    There exist constants $Z\geq 1$ and $Q>0$ such that for all $\alpha\in(Q,\infty)$ we have
    \begin{align*}
        \frac{1}{Z}\leq\frac{\sum_{l=1}^\infty\exp(-lq(\alpha)){l^{1-2b(\alpha)}}}{q(\alpha)^{-2+2b(\alpha)}}\leq Z
    \end{align*}
 
\end{prop}

\begin{proof}
      By Theorem \ref{Theorem Beardon}, we can take a small $\epsilon\in(0,1/2)$ such that $1/2<s-\epsilon$ and $s+\epsilon<1$. By Theorem \ref{Theorem main clain of previous paper}, there exists a constant $W>0$ such that for all $\alpha\in(W,\infty)$ we have $s-\epsilon<b(\alpha)$. Let $\alpha\in(W,\infty)$.
      For $q\in(0,\infty)$ we put 
      \begin{align*}
      K_{b(\alpha)}(q):= \sum_{l=1}^\infty\exp(-lq)l^{1-2b(\alpha)}.    
      \end{align*}
       The Mellin transform of $K_{b(\alpha)}(q)$ is $K^*_{b(\alpha)}(u)=\Gamma(u)\zeta(2b(\alpha)-1+u)$ (see \cite[Theorem 5 and Example 9]{FLAJOLET19953}). 
{$\Gamma(u)$ has a pole of order $1$ at $u=0$ and $\zeta(2b(\alpha)-1+u)$ has a pole of order $1$ at $u=2-2b(\alpha)$.} Thus,
by Mellin inversion theorem (see \cite[Theorem 2]{FLAJOLET19953}), for $D>2-2b(\alpha)$ we have
\begin{align*}
    K_{b(\alpha)}(q)=\frac{1}{2\pi i}\int_{D-i\infty}^{D+i\infty}\Gamma(u)\zeta(2b(\alpha)-1+u)q^{-u}du.
\end{align*}
 Put $\delta:=2\epsilon-1$. Note that $-1<\delta<0$ and for all $\alpha\in(W,\infty)$ we have $-1<2b(\alpha)-1+\delta<0$  since $\delta<2b(\alpha)-1+\delta<2(s-(1-\epsilon))<0$.
 {By the choice of $W$, we have $2-2b(\alpha)<1$.
 We fix $2-2b(\alpha)<D<1$.
By \cite[p.25]{FLAJOLET19953}, there exist constants $\beta, T, C>0$ such that for all $\delta\leq x\leq D$ and $|t|\geq T$ we have} 
\begin{align}\label{eq: convergence rate to 0}
  {|\Gamma(x+it)\zeta(2b(\alpha)-1+x+it)|\leq C\exp(-\beta |t|).}
\end{align}
Moreover, the residue of $\Gamma(u)$ at $u=0$ is $1$ and the residue of $\zeta(2b(\alpha)-1+u)$ at $u=2-2b(\alpha)$ is also $1$.
Therefore, shifting the integration line to the left, we obtain 
\begin{align}\label{eq in proposition}
    K_{b(\alpha)}(q)=&{\Gamma(2-2b(\alpha))}{q^{-2+2b(\alpha)}}+\zeta(2b(\alpha)-1)
    \\&\nonumber+\frac{1}{2i\pi}\int_{\delta-i\infty}^{\delta+i\infty}\Gamma(u)\zeta(2b(\alpha)-1+u)q(\alpha)^{-u}du.
\end{align}
By Lemma \ref{lemma asymptotic behavior of q}, Theorem \ref{Theorem main clain of previous paper} and (\ref{eq: convergence rate to 0}), we have 
\begin{align}\label{eq last term limit}
    \lim_{\alpha\to\infty}q(\alpha)^{2-2b(\alpha)}\left|\frac{1}{2i\pi}\int_{\delta-i\infty}^{\delta+i\infty}\Gamma(u)\zeta(2b(\alpha)-1+u)q(\alpha)^{-u}du
    \right|=0.
\end{align}
Note that for all $\alpha\in(W,\infty)$ we have  $0<2(s-\epsilon)-1<2b(\alpha)-1<2(1-\epsilon)-1<1$. By the continuity of $\zeta$ on $[2(s-\epsilon)-1,2(1-\epsilon)-1]$, Lemma \ref{lemma asymptotic behavior of q} and Theorem \ref{Theorem main clain of previous paper}, we have
\begin{align}\label{eq zeta}
    \lim_{\alpha\to\infty}q(\alpha)^{2-2b(\alpha)}|\zeta(2b(\alpha)-1)|=0. 
\end{align}
Since $0<2-2s<1$, we have $\Gamma(2-2s)>0$. We take $\eta>0$ such that $\eta<\Gamma(2-2s)$. Since $\lim_{\alpha\to\infty}b(\alpha)=s$ and $\Gamma$ is continuous on a small neighborhood of $2-2s$, there exists $Y>W$ such that for all $\alpha\in(Y,\infty)$ we have $\eta<\Gamma(2-2b(\alpha))$. 
 Therefore, by (\ref{eq in proposition}), (\ref{eq last term limit}) and (\ref{eq zeta}), the proof is complete.
\end{proof}

{By Lemma \ref{Lemma Dirichlet seires} and Proposition \ref{proposition main}, we obtain the following proposition.} 
\begin{prop}\label{prop q and alpha}
    {There exist constants $Z\geq 1$ and $Q>0$ such that for all $\alpha \in (Q,\infty)$ we have} 
    \begin{align*}
        {\frac{1}{Z}\leq \frac{q(\alpha)}{\alpha^{1/(-2+2b(\alpha))}}
        \leq Z.}
    \end{align*}
\end{prop}

\begin{pf}
By Lemma \ref{Lemma relationship q and b} and Proposition \ref{prop q and alpha}, there exist a constants $C\geq 1$ and $L>0$ such that for all $\alpha\in(L,\infty)$ we have
\begin{align}\label{eq comparable}
    \frac{1}{C}\leq\frac{s-b(\alpha)}{\int_{\alpha}^\infty t^{1/(-2+2b(t))}dt}\leq C.
\end{align}
By Theorem \ref{Theorem main clain of previous paper}, for all $x\leq 0$ we have $\lim_{\alpha\to\infty}(s-b(\alpha))\alpha^x=0$.

Let $x\in(0,{1}/({2-2s})-1)$. Then, there exists $\epsilon\in(0,2s-1)$ such that $x=1/(2-2s+\epsilon)-1$. By Theorem \ref{Theorem main clain of previous paper}, there exists a constant $W>L$ such that for all $\alpha\in(W,\infty)$ we have $s-\epsilon/2<b(\alpha)$. 
Since $\epsilon\in(0,2s-1)$, we have $1/2<s-\epsilon/2<1$ and 
\begin{align*}
\frac{1}{-2+2(s-\epsilon/2)}+1
=\frac{-1+2(s-\epsilon/2)}{-2+2(s-\epsilon/2)}<0.    
\end{align*}
Thus, since for all $\alpha\in(W,\infty)$ we have $1/(-2+2b(\alpha))<1/(-2+2(s-\epsilon/2))$, for all $\alpha\in(W,\infty)$ we obtain 
\begin{align*}
&\int_{\alpha}^\infty t^{{1}/{(-2+2b(t))}}dt
\leq\int_{\alpha}^\infty t^{{1}/{(-2+2(s-\epsilon/2))}}dt
\\&=-\left(\frac{1}{-2+2(s-\epsilon/2)}+1\right)\alpha^{{1}/({-2+2(s-\epsilon/2)})+1}.
\end{align*}
Hence, by (\ref{eq comparable}), we obtain $\lim_{\alpha\to\infty}(s-b(\alpha))\alpha^x=0$.

Let $x\in({1}/({2-2s})-1,\infty)$. There exists $\delta\in(0,\infty)$ such that $x={1}/({2-2s})-1+\delta$. 
By Theorem \ref{Theorem Beardon}, we have 
\begin{align*}
    \frac{1}{2s-2}-\frac{\delta}{2}+1=\frac{-1+2s-\delta(s-1)}{2s-2}<\frac{-1+2s}{2(s-1)}<0
\end{align*}
Thus, we obtain
\begin{align*}
&\int_{\alpha}^\infty t^{{1}/{(-2+2b(t))}}dt
\geq\int_{\alpha}^\infty t^{{1}/{(-2+2s)}-\delta/2}dt
\\&=-\left(\frac{1}{-2+2s}-\delta/2+1\right)\alpha^{{1}/{(-2+2s)}-\delta/2+1}.
\end{align*}
Hence, by (\ref{eq comparable}), we obtain $\lim_{\alpha\to\infty}(s-b(\alpha))\alpha^x=\infty$. \qed
\end{pf}

 \bibliographystyle{abbrv}
\bibliography{reference}

@book{katok1992fuchsian,
  title={Fuchsian groups},
  author={Katok, Svetlana},
  year={1992},
  publisher={University of Chicago press}
}

@book{beardon2012geometry,
  title={The geometry of discrete groups},
  author={Beardon, Alan F},
  volume={91},
  year={2012},
  publisher={Springer Science \& Business Media}
}

@book{Dal,
 title={Geodesic and horocyclic trajectories},
 author={F. Dal’Bo},
 year={2011},
 publisher={Springer-Verlag London, Ltd.London}
}

@book{walters2000introduction,
  title={An introduction to ergodic theory},
  author={Walters, Peter},
  volume={79},
  year={2000},
  publisher={Springer Science \& Business Media}
}

@book{mauldin2003graph,
  title={Graph directed {M}arkov systems: geometry and dynamics of limit sets},
  author={Mauldin, R Daniel and Urba{\'n}ski, Mariusz},
  volume={148},
  year={2003},
  publisher={Cambridge University Press}
}

@book{falconer2004fractal,
  title={Fractal geometry: mathematical foundations and applications},
  author={Falconer, Kenneth},
  year={2004},
  publisher={John Wiley \& Sons}
}

@book{przytycki2010conformal,
  title={Conformal fractals: ergodic theory methods},
  author={Przytycki, Feliks and Urba{\'n}ski, Mariusz},
  volume={371},
  year={2010},
  publisher={Cambridge University Press}
}

@article{bowen1979markov,
  title={Markov maps associated with {F}uchsian groups},
  author={Bowen, Rufus and Series, Caroline},
  journal={Publications Math{\'e}matiques de l'IH{\'E}S},
  volume={50},
  pages={153--170},
  year={1979}
}

@article{rush2023multifractal,
  title={Multifractal analysis for {M}arkov interval maps with countably many branches},
  author={Rush, Tom},
  journal={Nonlinearity},
  volume={36},
  number={4},
  pages={2038},
  year={2023},
  publisher={IOP Publishing}
}

@article{iommi2015multifractal,
  title={Multifractal analysis of {B}irkhoff averages for countable {M}arkov maps},
  author={Iommi, Godofredo and Jordan, Thomas},
  journal={Ergodic Theory and Dynamical Systems},
  volume={35},
  number={8},
  pages={2559--2586},
  year={2015},
  publisher={Cambridge University Press}
}

@article{jaerisch2021mixed,
  title={Mixed multifractal spectra of {B}irkhoff averages for non-uniformly expanding one-dimensional {M}arkov maps with countably many branches},
  author={Jaerisch, Johannes and Takahasi, Hiroki},
  journal={Advances in Mathematics},
  volume={385},
  pages={107778},
  year={2021},
  publisher={Elsevier}
}

@book{borthwick2007spectral,
  title={Spectral theory of infinite-area hyperbolic surfaces},
  author={Borthwick, David},
  year={2007},
  publisher={Springer}
}

@article{Jaerisch2016AMA,
  title={A multifractal analysis for cuspidal windings on hyperbolic surfaces},
  author={Johannes Jaerisch and Marc Kessebohmer and Sara Munday},
  journal={Stochastics and Dynamics},
  year={2016},
  url={https://api.semanticscholar.org/CorpusID:54513074}
}

@article{series1985modular,
  title={The modular surface and continued fractions},
  author={Series, Caroline},
  journal={Journal of the London Mathematical Society},
  volume={2},
  number={1},
  pages={69--80},
  year={1985},
  publisher={Oxford University Press}
}

@article{Munday2011OnHD,
  title={On {H}ausdorff dimension and cusp excursions for Fuchsian groups},
  author={Sara Munday},
  journal={Discrete Contin. Dyn. Syst},
  year={2012}
}

@book{urubanskinoninvertible ,
  title={Non-Invertible Dynamical Systems: Volume 2 Finer Thermodynamic Formalism–Distance Expanding Maps and Countable State Subshifts of Finite Type, Conformal GDMSs, Lasota-Yorke Maps and Fractal Geometry },
  author={Urba{\'n}ski, Mariusz and Roy, Mario and Munday, Sara},
  volume={490},
  year={2022},
  publisher={De Gruyter Expositions in Mathematics}
}

@book{kato2013perturbation,
  title={Perturbation theory for linear operators},
  author={Kato, Tosio},
  volume={132},
  year={2013},
  publisher={Springer Science \& Business Media}
}

@article{jaerisch2022multifractal,
  title={Multifractal analysis of homological growth rates for hyperbolic surfaces},
  author={Jaerisch, Johannes and Takahasi, Hiroki},
  journal={arXiv preprint arXiv:2204.08907},
  year={2022}
}

@article{FLAJOLET19953,
title = {Mellin transforms and asymptotics: Harmonic sums},
author = {Philippe Flajolet and Xavier Gourdon and Philippe Dumas},
journal = {Theoretical Computer Science},
volume = {144},
number = {1},
pages = {3-58},
year = {1995},
}

@article{arima2024higher,
  title={Higher-dimensional multifractal analysis for the cusp winding process on hyperbolic surfaces},
  author={Arima, Yuya},
  journal={arXiv preprint arXiv:2402.16418},
  year={2024}
}

@article{cesaratto2006hausdorff,
  title={Hausdorff dimension of real numbers with bounded digit averages},
  author={Cesaratto, Eda and Vall{\'e}e, Brigitte},
  journal={Acta Arithmetica},
  pages={115--162},
  year={2006}
}

@article{fan2009khintchine,
  title={On Khintchine exponents and Lyapunov exponents of continued fractions},
  author={Fan, Ai-Hua and Liao, Ling-Min and Wang, Bao-Wei and Wu, Jun},
  journal={Ergodic Theory and Dynamical Systems},
  volume={29},
  number={1},
  pages={73--109},
  year={2009},
  publisher={Cambridge University Press}
}

@article{mauldin1999conformal,
  title={Conformal iterated function systems with applications to the geometry of continued fractions},
  author={Mauldin, R and Urba{\'n}ski, Mariusz},
  journal={Transactions of the American Mathematical Society},
  volume={351},
  number={12},
  pages={4995--5025},
  year={1999}
}

@article{hensley1992continued,
  title={Continued fraction Cantor sets, Hausdorff dimension, and functional analysis},
  author={Hensley, Doug},
  journal={Journal of number theory},
  volume={40},
  number={3},
  pages={336--358},
  year={1992},
  publisher={Elsevier}
}

@book{barreira2008dimension,
  title={Dimension and recurrence in hyperbolic dynamics},
  author={Barreira, Luis},
  volume={272},
  year={2008},
  publisher={Springer}
}

@book{pesin2008dimension,
  title={Dimension theory in dynamical systems: contemporary views and applications},
  author={Pesin, Yakov B},
  year={2008},
  publisher={University of Chicago Press}
}
 \nocite{*}

\end{document}